\documentclass[a4paper,11pt]{article}
 
\usepackage{latexsym,amssymb,amsmath,amsthm,mathrsfs}  

\usepackage[width=14.5cm]{geometry}
 
\newtheorem{thm}{Theorem}
\newtheorem{lem}{Lemma}
\newtheorem{prop}{Proposition}
\newtheorem{cor}{Corollary}
\newtheorem{defn}{Definition} 
\newtheorem{rem}{Remark}

\begin{document}

\title{On two-sided monogenic functions of axial type\thanks{accepted for publication in Moscow Mathematical Journal}}

\author{Dixan Pe\~na Pe\~na$^{\text{a}}$\\
\small{e-mail: dixanpena@gmail.com}
\and Irene Sabadini$^{\text{a}}$\\
\small{e-mail: irene.sabadini@polimi.it}
\and Frank Sommen$^{\text{b}}$\\
\small{e-mail: fs@cage.ugent.be}}

\date{\small{$^\text{a}$Dipartimento di Matematica, Politecnico di Milano\\Via E. Bonardi 9, 20133 Milano, Italy\\\vspace{0.2cm}
$^{\text{b}}$Clifford Research Group, Department of Mathematical Analysis\\Faculty of Engineering and Architecture, Ghent University\\Galglaan 2, 9000 Gent, Belgium}}

\maketitle

\begin{abstract}
\noindent In this paper we study two-sided (left and right) axially symmetric solutions of a generalized Cauchy-Riemann operator. We present three methods to obtain special solutions: via the Cauchy-Kowalevski extension theorem, via plane wave integrals and Funk-Hecke's formula and via primitivation. Each of these methods is effective enough to generate all the polynomial solutions.\vspace{0.2cm}\vspace{0.1cm}
\end{abstract}

\section{Introduction}

Let  $\mathbb{R}_{0,m}$ be the real Clifford algebra generated by the canonical basis $\{e_1,\ldots,e_m\}$ of the Euclidean space $\mathbb R^m$ (see \cite{Cl,Lo}). It is an associative algebra in which the multiplication has the property $\underline x^2=-\vert\underline x\vert^2=-\sum_{j=1}^mx_j^2$ for any $\underline x=\sum_{j=1}^mx_je_j\in\mathbb R^m$. This requirement clearly implies the following multiplication rules
\[e_je_k+e_ke_j=-2\delta_{jk},\quad j,k\in\{1,\dots,m\}.\]
Any Clifford number $a\in\mathbb R_{0,m}$ may thus be written as 
\[a=\sum_Aa_Ae_A,\quad a_A\in\mathbb R,\] 
using the basis elements $e_A=e_{j_1}\dots e_{j_k}$ defined for every subset $A=\{j_1,\dots,j_k\}$ of $\{1,\dots,m\}$ with $j_1<\dots<j_k$ (for $A=\emptyset$ one puts $e_{\emptyset}=1$). Conjugation in $\mathbb R_{0,m}$ is given by $\overline a=\sum_Aa_A\overline e_A$, where $\overline e_A=\overline e_{j_k}\dots\overline e_{j_1}$, $\overline e_j=-e_j$, $j=1,\dots,m$. It is easy to check that 
\begin{equation}\label{revconj}
\overline{ab}=\overline b\overline a,\quad a,b\in\mathbb R_{0,m}.
\end{equation}
For each $\ell\in\{0,1,\dots,m\}$ we call
\[\mathbb R_{0,m}^{(\ell)}=\text{span}_{\mathbb R}\big(e_A:\;\vert A\vert=\ell\big)\] 
the subspace of $\ell$-vectors, i.e. the subspace spanned by the products of $\ell$ different basis vectors. Thus, every element $a\in\mathbb R_{0,m}$ admits the so-called multivector decomposition 
\[a=\sum_{\ell=0}^m[a]_{\ell},\] 
where $[a]_{\ell}$ denotes the projection of $a$ on $\mathbb R_{0,m}^{(\ell)}$. 

Observe that the product of two Clifford vectors $\underline x=\sum_{j=1}^mx_je_j$ and $\underline y=\sum_{j=1}^my_je_j$ splits into a scalar part and a 2-vector part
\begin{equation*}
\underline x\,\underline y=\underline x\bullet\underline y+\underline x\wedge\underline y\in\mathbb R_{0,m}^{(0)}\oplus\mathbb R_{0,m}^{(2)},
\end{equation*}
where
\[\underline x\bullet\underline y=-\left\langle\underline x,\underline y\right\rangle=-\sum_{j=1}^mx_jy_j\]
equals, up to a minus sign, the standard Euclidean inner product between $\underline x$ and $\underline y$, while
\[\underline x\wedge\underline y=\sum_{j=1}^m\sum_{k=j+1}^me_je_k(x_jy_k-x_ky_j)\] 
represents the standard outer (or wedge) product between them.

One natural way to extend the theory of holomorphic functions of a complex variable to higher dimensions is to consider the null solutions of the so-called generalized Cauchy-Riemann operator in $\mathbb R^{m+1}$, given by
\[\partial_{x_0}+\partial_{\underline x},\]
where $\partial_{\underline x}=\sum_{j=1}^me_j\partial_{x_j}$ is the Dirac operator in $\mathbb R^m$ (see \cite{BDS,CS4,DSS,GM,GuSp}).

\begin{defn}
A function $f:\Omega\rightarrow\mathbb{R}_{0,m}$ defined and continuously differentiable in an open set $\Omega$ in $\mathbb R^{m+1}$ is said to be left {\rm(}resp. right{\rm)} monogenic in $\Omega$ if $(\partial_{x_0}+\partial_{\underline x})f=0$ {\rm(}resp. $f(\partial_{x_0}+\partial_{\underline x})=0${\rm)} in $\Omega$. Moreover, functions which are both left and right monogenic, i.e. functions satisfying the overdetermined system 
\begin{equation}\label{TSidedEq}
(\partial_{x_0}+\partial_{\underline x})f=f(\partial_{x_0}+\partial_{\underline x})=0,
\end{equation}
are called two-sided monogenic.
\end{defn}

In a similar fashion is defined monogenicity with respect to the Dirac operator $\partial_{\underline x}$. Note that the differential operator $\partial_{x_0}+\partial_{\underline x}$ provides a factorization of the Laplacian in the sense that
\[\Delta=\sum_{j=0}^m\partial_{x_j}^2=(\partial_{x_0}+\partial_{\underline x})(\partial_{x_0}-\partial_{\underline x})\]
and hence monogenic functions are harmonic. 

One basic yet fundamental result in Clifford analysis is the Cauchy-Kowalevski extension theorem, which states that every monogenic function in $\mathbb R^{m+1}$ is determined by its restriction to $\mathbb R^m$ (see \cite{So1}).

\begin{thm}[Cauchy-Kowalevski extension theorem]\label{CKextThm}
Every function $g(\underline x)$ analytic in the open set $\,\underline\Omega\subset\mathbb R^m$ has a unique left monogenic extension given by
\begin{equation*}\label{CKf}
\mathsf{CK}[g(\underline x)](x_0,\underline x)=\sum_{n=0}^\infty\frac{(-x_0)^n}{n!}\,\partial_{\underline x}^ng(\underline x),
\end{equation*}
and defined in an open neighbourhood $\Omega\subset\mathbb R^{m+1}$ of $\,\underline\Omega$.
\end{thm}

This result leads to the construction of special monogenic functions depending on the choice of the initial function $g(\underline x)$. For instance, if $g$ is a function of the variable $\langle\underline x,\underline t\rangle$ with $\underline t\in\mathbb R^m$ fixed, then $\mathsf{CK}[g(\langle\underline x,\underline t\rangle)]$ will produce a so-called monogenic plane wave function (see \cite{So2,So3}). 

Let us denote by $\mathsf{M}_{l}(k)$ (resp. $\mathsf{M}_{r}(k)$) the set of all left (resp. right) monogenic homogeneous polynomials of degree $k$ in $\mathbb R^m$. Another class of special monogenic functions we shall deal in this paper is the class of axial left monogenic functions (see \cite{LB,S1,S2,S3}). They are left monogenic functions of the form 
\begin{equation}\label{AxialLMF}
\left(M(x_0,r)+\frac{\underline x}{r}\,N(x_0,r)\right)P_k(\underline x),\quad r=\vert\underline x\vert,
\end{equation}
where $M$, $N$ are $\mathbb R$-valued continuously differentiable functions depending on the two variables $(x_0,r)$ and $P_k(\underline x)$ belongs to $\mathsf{M}_{l}(k)$. It can be easily shown that $M$ and $N$ must satisfy the following Vekua-type system (see \cite{Ve})  
\begin{equation}\label{VeEq}
\left\{\begin{aligned}
\partial_{x_0}M-\partial_rN&=\frac{2k+m-1}{r}N\\
\partial_rM+\partial_{x_0}N&=0.
\end{aligned}\right.
\end{equation}
One may prove that every left monogenic homogeneous polynomial $M_k(x_0,\underline x)$ of degree $k$ in $\mathbb R^{m+1}$ can be expressed as a finite sum of axial left monogenic functions, i.e.
\begin{equation}\label{maindecomp}
M_k(x_0,\underline x)=\sum_{n=0}^k\mathsf{CK}\left[\underline x^nP_{k-n}(\underline x)\right](x_0,\underline x),\quad P_{k-n}(\underline x)\in\mathsf{M}_{l}(k-n),
\end{equation}
and thus showing that the axial left monogenic functions are in fact the building blocks of the solutions of the equation $(\partial_{x_0}+\partial_{\underline x})f=0$.

The analogues of functions (\ref{AxialLMF}) for the case of two-sided monogenicity were introduced in \cite{DSo} and are defined as follows.

\begin{defn}\label{a2sidedm}
Let $P_{k,\ell}(\underline x)$ be an $\mathbb R_{0,m}^{(\ell)}$-valued polynomial belonging to $\mathsf{M}_{l}(k)$ $(1\le\ell\le m-1)$. A function is called axial two-sided monogenic if it is two-sided monogenic and is of the form  
\begin{equation}\label{AxialTSMF}
A(x_0,r)P_{k,\ell}(\underline x)+B(x_0,r)\underline xP_{k,\ell}(\underline x)+C(x_0,r)P_{k,\ell}(\underline x)\underline x+D(x_0,r)\underline x P_{k,\ell}(\underline x)\underline x,
\end{equation}
where $r=\vert\underline x\vert$ and $A$, $B$, $C$, $D$ are $\mathbb R$-valued continuously differentiable functions in some open subset of $\,\mathbb R^2_+=\{(x_1,x_2)\in\mathbb R^2:\;x_2>0\}$.   
\end{defn}

In order to allow for explicit computations we assume that $P_{k,\ell}$ takes values in the subspace of $\ell$-vectors. Note that this assumption implies that $P_{k,\ell}$ is two-sided monogenic.  Indeed, from $\partial_{\underline x}P_{k,\ell}=0$ and using (\ref{revconj}) we obtain 
\[0=\overline{P_{k,\ell}}\partial_{\underline x}=(-1)^{\frac{\ell(\ell+1)}{2}}P_{k,\ell}\partial_{\underline x}.\]
It thus follows that $P_{k,\ell}\partial_{\underline x}=0$. The consideration of functions (\ref{AxialTSMF}) leads to a system of first-order partial differential equations with variable coefficients (see \cite{DSo}).

\begin{prop}\label{caract1}
A function is axial two-sided monogenic if and only if $C=B$ and 
\begin{equation}\label{Veq2sided}
\left\{\begin{aligned}
\partial_{x_0}A-r\partial_rB&=\left(2k+m-\mu_{\ell}\right)B\\
\partial_{x_0}B+\frac{1}{r}\partial_rA&=\mu_{\ell}D\\
\partial_{x_0}B-r\partial_rD&=(2k+m+2)D\\
\partial_{x_0}D+\frac{1}{r}\partial_rB&=0,
\end{aligned}\right.
\end{equation} 
where $\mu_{\ell}=(-1)^{\ell}(2\ell-m)$.
\end{prop}

In this paper we study axial two-sided monogenic functions in a neighbourhood of the origin. Each such function admits a Taylor series decomposition in terms of two-sided monogenic polynomials that are of axial type. 

In Section \ref{secc2} we give a characterization of such two-sided monogenic polynomials in terms of the Cauchy-Kowalevski extension theorem. In particular we characterize those polynomials for which the CK-extension will be axial two-sided monogenic and prove that this class of polynomials spans the space of all polynomials two-sided monogenics.   

In Section \ref{secc3} we consider two-sided monogenic plane waves. They depend on a parameter $\underline t\in S^{m-1}$ and after integrating over the unit sphere $S^{m-1}$ and applying Funk-Hecke's formula one obtains axial two-sided monogenics. We show that all polynomial axial two-sided monogenics may be obtained as integrals of such plane waves. We also construct axial two-sided monogenics that are expressed in terms of Bessel functions.

In the final Section \ref{secc4} we start from the simple observation that if 
\[f(x_0,\underline x)=\left(M(x_0,r)+\displaystyle{\frac{\underline x}{r}}\,N(x_0,r)\right)P_{k,\ell}(\underline x)\] 
is axial left monogenic, then $f(x_0,\underline x)(\partial_{x_0}-\partial_{\underline x})$ is axial two-sided monogenic. We prove that all axial two-sided monogenics may locally be obtained in this way.

So we have several methods to obtain polynomials solutions. Of course one can also consider axial two-sided monogenics in more general domains with possible singularities on the axis or in the origin. It remains to be studied how such solutions might be obtained from the methods exposed here.
 
A method for obtaining polynomial solutions to the Hodge-de Rham system was obtained in \cite{DLaS}. Although the Hodge-de Rham system can be seen as a two-sided monogenic system with respect to the Dirac operator $\partial_{\underline x}$, the authors do not use Vekua systems (see \cite{DSo}), Bessel functions and plane wave integrals.   

\section{Homogeneous two-sided monogenic polynomials in $\mathbb R^{m+1}$}\label{secc2}

The aim of this section is to prove an analogue of the decomposition (\ref{maindecomp}) for the case of two-sided monogenic homogeneous polynomials in $\mathbb R^{m+1}$. 

We begin by observing that
\[e_je_Ae_j=\left\{\begin{array}{ll}(-1)^{\vert A\vert}e_A&\text{for}\quad j\in A,\\(-1)^{\vert A\vert+1}e_A&\text{for}\quad j\notin A,\end{array}\right.\]
which clearly yields $\sum_{j=1}^me_je_Ae_j=(-1)^{\vert A\vert}(2\vert A\vert-m)e_A$. Therefore for every $a\in\mathbb R_{0,m}^{(\ell)}$ the following equality holds
\begin{equation}\label{trickyeq}
\sum_{j=1}^me_jae_j=\mu_{\ell}a,\quad\mu_{\ell}=(-1)^{\ell}(2\ell-m).
\end{equation}
The fact that polynomial $P_{k,\ell}(\underline x)$ in Definition \ref{a2sidedm} is two-sided monogenic remains valid for every left monogenic function $F(\underline x)$ with values in $\mathbb R_{0,m}^{(\ell)}$. We can say even more: $F(\underline x)$ is two-sided monogenic if and only if $[F(\underline x)]_{\ell}$ is left monogenic for $\ell=0,\dots,m$ (see e.g. \cite{ABoDS}). For the sake of completeness we include a proof here.

\begin{prop}\label{deco2sided}
Consider the multivector decomposition of function $F(\underline x)$, i.e. 
\[F=\sum_{\ell=0}^m[F]_{\ell}.\]
Then $F$ is two-sided monogenic if and only if each $[F]_{\ell}$ is left monogenic.
\end{prop}
\begin{proof}
We have already seen that the condition is sufficient so we have to prove only the necessity. Put $F_{\ell}=[F]_{\ell}$. Observe that  $\partial_{\underline x}F_{\ell}$ decomposes into a $(\ell-1)$-vector and a $(\ell+1)$-vector, i.e.
\[\partial_{\underline x}F_{\ell}=\left[\partial_{\underline x}F_{\ell}\right]_{\ell-1}+\left[\partial_{\underline x}F_{\ell}\right]_{\ell+1}.\] 
Hence $F$ satisfies $\partial_{\underline x}F=0$ if and only if 
\begin{equation}\label{condizq}
\left[\partial_{\underline x}F_{\ell-1}\right]_{\ell}+\left[\partial_{\underline x}F_{\ell+1}\right]_{\ell}=0,\quad\ell=0,\dots,m,
\end{equation}
with $F_{-1}=F_{m+1}=0$. Similarly, $F$ is right monogenic if and only if 
\[\left[F_{\ell-1}\partial_{\underline x}\right]_{\ell}+\left[F_{\ell+1}\partial_{\underline x}\right]_{\ell}=0,\quad\ell=0,\dots,m,\]
or equivalently
\begin{equation}\label{condder}
\left[\partial_{\underline x}F_{\ell-1}\right]_{\ell}-\left[\partial_{\underline x}F_{\ell+1}\right]_{\ell}=0,\quad\ell=0,\dots,m,
\end{equation}
where we have used the identities 
\[\left[F_{\ell-1}\partial_{\underline x}\right]_{\ell}=(-1)^{\ell-1}\left[\partial_{\underline x}F_{\ell-1}\right]_{\ell},\quad\left[F_{\ell+1}\partial_{\underline x}\right]_{\ell}=(-1)^{\ell}\left[\partial_{\underline x}F_{\ell+1}\right]_{\ell}.\]
It follows from (\ref{condizq}) and (\ref{condder}) that $\left[\partial_{\underline x}F_{\ell-1}\right]_{\ell}=\left[\partial_{\underline x}F_{\ell+1}\right]_{\ell}=0$. This clearly ensures that each $F_{\ell}$ is left monogenic.
\end{proof}

\begin{rem}
The scalar part $[F]_{0}$ and the pseudoscalar part $[F]_{m}$ of a two-sided monogenic function defined in an open connected subset of $\,\mathbb R^m$ are constants.
\end{rem}

In what follows, we recall some essential identities. Let $A$, $B$, $C$, $D$ and $P_{k,\ell}$ be as in Definition \ref{a2sidedm}. It is easily seen that 
\[\partial_{\underline x}A=\sum_{j=1}^me_j\partial_{x_j}A=\sum_{j=1}^me_j(\partial_rA)(\partial_{x_j}r)=\frac{\partial_rA}{r}\,\underline x\]
and therefore
\begin{equation}\label{ident1}
\partial_{\underline x}\big(AP_{k,\ell}\big)=(\partial_{\underline x}A)P_{k,\ell}+A\partial_{\underline x}P_{k,\ell}=\frac{\partial_rA}{r}\underline xP_{k,\ell}.
\end{equation}
Using the identity $\partial_{\underline x}(\underline xf)=-mf-2\sum_{j=1}^mx_j\partial_{x_j}f-\underline x\partial_{\underline x}f$ and Euler's theorem for homogeneous functions, we also obtain that
\begin{multline}\label{ident2}
\partial_{\underline x}\big(B\underline xP_{k,\ell}\big)=(\partial_rB)\frac{\underline x^2}{r}P_{k,\ell}-B\Big(mP_{k,\ell}+2\sum_{j=1}^mx_j\partial_{x_j}P_{k,\ell}+\underline x\partial_{\underline x}P_{k,\ell}\Big)\\
=-\big((2k+m)B+r\partial_rB\big)P_{k,\ell}. 
\end{multline}
On account of (\ref{trickyeq}) we get
\[\partial_{\underline x}\big(P_{k,\ell}\underline x\big)=\left(\partial_{\underline x}P_{k,\ell}\right)\underline x+\sum_{j=1}^me_jP_{k,\ell}(\partial_{x_j}\underline x)=\mu_{\ell}P_{k,\ell}.\]
This gives
\begin{align}
\partial_{\underline x}\big(CP_{k,\ell}\underline x\big)&=\mu_{\ell}CP_{k,\ell}+\frac{\partial_rC}{r}\underline xP_{k,\ell}\underline x,\\
\partial_{\underline x}\big(D\underline xP_{k,\ell}\underline x\big)&=-\mu_{\ell}D\underline xP_{k,\ell}-\big((2k+m+2)D+r\partial_rD\big)P_{k,\ell}\underline x.\label{ident3-4}
\end{align}
In the same way we can deduce identities for 
\[\big(AP_{k,\ell}\big)\partial_{\underline x},\quad\big(B\underline xP_{k,\ell}\big)\partial_{\underline x},\quad\big(CP_{k,\ell}\underline x\big)\partial_{\underline x}\quad\text{and}\quad\big(D\underline xP_{k,\ell}\underline x\big)\partial_{\underline x}.\]

\begin{lem}\label{lemfund}
Assume that $R_n(\underline x),S_n(\underline x)\in\mathsf{M}_{l}(n)\cap\mathsf{M}_{r}(n)$ for $n=0,\dots,k$ and let $S_k=0$. If 
\begin{equation}\label{FraIguald}
\sum_{\substack{n=0\\n\;{\rm even}}}^k\left(\vert\underline x\vert^{n}R_{k-n}+\vert\underline x\vert^{n-2}\underline xS_{k-n}\underline x\right)+\sum_{\substack{n=1\\n\;{\rm odd}}}^k\vert\underline x\vert^{n-1}\left(\underline xR_{k-n}+S_{k-n}\underline x\right)=0,
\end{equation}
then all polynomials $R_n,S_n$ are identically equal to zero, except possibly $R_0$ and $S_0$. More precisely 
\begin{alignat*}{2}
R_n&=S_n=0,&\quad&n=1,\dots,k,\\
[R_0]_{\ell}&=[S_0]_{\ell}=0,&\quad&\ell=1,\dots,m-1,\\
[R_0]_{0}&=(-1)^k[S_0]_{0},&\quad&[R_0]_{m}=(-1)^{m+k-1}[S_0]_{m}.
\end{alignat*}
\end{lem}
\begin{proof}
We shall prove the assertion by induction. When $k=1$ we have
\[R_1+\underline xR_0+S_0\underline x=0,\]
from which we obtain 
\begin{align*}
0&=\partial_{\underline x}(R_1+\underline xR_0+S_0\underline x)=-mR_0+\sum_{\ell=0}^m\mu_{\ell}[S_0]_{\ell},\\
0&=(R_1+\underline xR_0+S_0\underline x)\partial_{\underline x}=\sum_{\ell=0}^m\mu_{\ell}[R_0]_{\ell}-mS_0
\end{align*}
and hence
\begin{equation*}
\left\{\begin{array}{ll}m[R_0]_{\ell}-\mu_{\ell}[S_0]_{\ell}&=0\\\mu_{\ell}[R_0]_{\ell}-m[S_0]_{\ell}&=0.\end{array}\right.
\end{equation*}
It thus follows that 
\begin{alignat*}{2}
[R_0]_{\ell}&=[S_0]_{\ell}=0,&\quad&\ell=1,\dots,m-1,\\ 
[R_0]_{0}&=-[S_0]_{0},&\quad&[R_0]_{m}=(-1)^{m}[S_0]_{m},
\end{alignat*}
showing also that $\underline xR_0+S_0\underline x=0$ and therefore $R_1=0$. The statement is then true for $k=1$. 

Now we proceed to show that if the assertion holds for some positive integer $k\ge1$, then it also holds $k+1$. First, note that for $k+1$ equality (\ref{FraIguald}) may be rewritten as
\[R_{k+1}+\sum_{\substack{n=0\\n\;\text{even}}}^k\vert\underline x\vert^{n}\left(\underline xR_{k-n}+S_{k-n}\underline x\right)+\sum_{\substack{n=1\\n\;\text{odd}}}^k\left(\vert\underline x\vert^{n+1}R_{k-n}+\vert\underline x\vert^{n-1}\underline xS_{k-n}\underline x\right)=0.\]
Letting the Dirac operator $\partial_{\underline x}$ act from the left on the last equality, we obtain
\begin{multline*}
\sum_{\substack{n=0\\n\;\text{even}}}^k\bigg(\vert\underline x\vert^{n}\sum_{\ell}\Big(\mu_{\ell}[S_{k-n}]_{\ell}-(2k+m-n)[R_{k-n}]_{\ell}\Big)+n\vert\underline x\vert^{n-2}\underline xS_{k-n}\underline x\bigg)\\
+\sum_{\substack{n=1\\n\;\text{odd}}}^k\vert\underline x\vert^{n-1}\bigg(\underline x\sum_{\ell}\Big((n+1)[R_{k-n}]_{\ell}-\mu_{\ell}[S_{k-n}]_{\ell}\Big)-(2k+m-n+1)S_{k-n}\underline x\bigg)=0,
\end{multline*}
where we have used identities (\ref{ident1})-(\ref{ident3-4}). On account of Proposition \ref{deco2sided} and since we have assumed that the assertion is true for $k$, it easily follows from the last equality that
\begin{equation}\label{twinid}
(2k+m)[R_k]_{\ell}-\mu_{\ell}[S_k]_{\ell}=0
\end{equation}
and
\begin{alignat*}{2}
R_n&=S_n=0,&\quad&n=1,\dots,k-1,\\
[R_0]_{\ell}&=[S_0]_{\ell}=0,&\quad&\ell=1,\dots,m-1,\\
[R_0]_{0}&=(-1)^{k+1}[S_0]_{0},&\quad&[R_0]_{m}=(-1)^{m+k}[S_0]_{m}.
\end{alignat*}
These equalities imply that $R_{k+1}+\underline x R_k+S_k\underline x=0$. If we now let $\partial_{\underline x}$ act from the right, then we get
\[\mu_{\ell}[R_k]_{\ell}-(2k+m)[S_k]_{\ell}=0,\]
which together with (\ref{twinid}) clearly implies that $[R_k]_{\ell}=[S_k]_{\ell}=0$ and hence $R_{k+1}=R_k=S_k=0$.
\end{proof}

We next recall two fundamental decompositions for homogeneous polynomials. The first one is the classical Fischer decomposition in terms of harmonic homogeneous polynomials while the second one is given using two-sided monogenic homogeneous polynomials (see e.g. \cite{DSS}). 

\begin{thm}[Fischer decompositions]\label{Fisdecomp}
Let $\mathsf{P}(k)$ be the set of all homogeneous polynomials of degree $k$ in $\mathbb R^m$. By $\mathsf{H}(k)$ we denote the polynomials in $\mathsf{P}(k)$ which are harmonic. If $P_k(\underline x)\in\mathsf{P}(k)$, then the following two decompositions hold:
\begin{align*}
P_k&=H_k+\vert\underline x\vert^2P_{k-2},\quad H_k\in\mathsf{H}(k),\;P_{k-2}\in\mathsf{P}(k-2),\\
P_k&=M_k+\underline xP_{k-1}+Q_{k-1}\underline x,\quad M_k\in\mathsf{M}_{l}(k)\cap\mathsf{M}_{r}(k),\;P_{k-1},Q_{k-1}\in\mathsf{P}(k-1).
\end{align*}
\end{thm}
\noindent
Before proving the main result of the section it is useful to notice the following.

\begin{rem}
 An analytic function $g(\underline x)$ has a two-sided monogenic extension if and only if it satisfies the condition $\partial_{\underline x}g=g\partial_{\underline x}$. Indeed, if $f(x_0,\underline x)$ is a two-sided monogenic extension of $g$, then from {\rm(}\ref{TSidedEq}{\rm)} it follows that $\partial_{\underline x}f=f\partial_{\underline x}$ and hence $\partial_{\underline x}g=g\partial_{\underline x}$. Finally, observe that this condition implies that $\mathsf{CK}[g(\underline x)]$ is two-sided monogenic. 
\end{rem}

\begin{thm}\label{desc2sided}
Suppose that $M_k(x_0,\underline x)$ is a two-sided monogenic homogeneous polynomial of degree $k$ in $\mathbb R^{m+1}$. Then there exist polynomials $S_n(\underline x)\in\mathsf{M}_{l}(n)\cap\mathsf{M}_{r}(n)$, $n=0,\dots,k$, such that
\begin{multline*}
M_k(x_0,\underline x)=S_k(\underline x)+\sum_{\substack{n=1\\n\;{\rm odd}}}^k\mathsf{CK}\Big[\vert\underline x\vert^{n-1}\big(\underline xS_{k-n}(\underline x)+S_{k-n}(\underline x)\underline x\big)\Big](x_0,\underline x)\\
+\sum_{\substack{n=2\\n\;{\rm even}}}^k\sum_{\ell}\mathsf{CK}\Big[\lambda_{n,\ell}\vert\underline x\vert^{n}[S_{k-n}(\underline x)]_{\ell}+\vert\underline x\vert^{n-2}\underline x[S_{k-n}(\underline x)]_{\ell}\underline x\Big](x_0,\underline x),
\end{multline*}
where $\lambda_{n,\ell}=-\displaystyle{\frac{(2k+m-n-\mu_{\ell})}{n}}$.
\end{thm}
\begin{proof}
By Theorem \ref{CKextThm} we have that $M_k(x_0,\underline x)=\mathsf{CK}[M_k(0,\underline x)](x_0,\underline x)$. As $M_k(0,\underline x)\in\mathsf{P}(k)$ it follows from the second Fischer decomposition of Theorem \ref{Fisdecomp} that 
\[M_k(0,\underline x)=\sum_{n_1=0}^k\sum_{n_2=0}^{n_1}\underline x^{n_1-n_2}M_{k-n_1,n_2}(\underline x)\underline x^{n_2},\] 
where $M_{k-n_1,n_2}(\underline x)\in\mathsf{M}_{l}(k-n_1)\cap\mathsf{M}_{r}(k-n_1)$. Observe that  $\underline x^{n_1-n_2}M_{k-n_1,n_2}\underline x^{n_2}$ may be rewritten as
\[(-1)^{\frac{n_1}{2}}\vert\underline x\vert^{n_1}M_{k-n_1,n_2}\quad\text{or}\quad(-1)^{\frac{n_1-2}{2}}\vert\underline x\vert^{n_1-2}\underline xM_{k-n_1,n_2}\underline x,\]
for $n_1$ even, while for $n_1$ odd $\underline x^{n_1-n_2}M_{k-n_1,n_2}\underline x^{n_2}$ equals
\[(-1)^{\frac{n_1-1}{2}}\vert\underline x\vert^{n_1-1}\underline xM_{k-n_1,n_2}\quad\text{or}\quad(-1)^{\frac{n_1-1}{2}}\vert\underline x\vert^{n_1-1}M_{k-n_1,n_2}\underline x.\]
Therefore, there exist $R_n(\underline x),S_n(\underline x)\in\mathsf{M}_{l}(n)\cap\mathsf{M}_{r}(n)$ so that
\begin{multline*}
M_k(0,\underline x)=R_k(\underline x)+\sum_{\substack{n=1\\n\;\text{odd}}}^k\vert\underline x\vert^{n-1}\big(\underline xR_{k-n}(\underline x)+S_{k-n}(\underline x)\underline x\big)\\
+\sum_{\substack{n=2\\n\;\text{even}}}^k\big(\vert\underline x\vert^{n}R_{k-n}(\underline x)+\vert\underline x\vert^{n-2}\underline xS_{k-n}(\underline x)\underline x\big).
\end{multline*}
Note that $M_k(0,\underline x)$ must satisfy the condition $\partial_{\underline x}M_k(0,\underline x)=M_k(0,\underline x)\partial_{\underline x}$ since $M_k(x_0,\underline x)$ is two-sided monogenic. We thus get 
\begin{multline*}
\sum_{\substack{n=0\\n\;\text{even}}}^{k-1}\bigg(\vert\underline x\vert^{n}\sum_{\ell}a_{n,\ell}\Big([S_{k-n-1}]_{\ell}-[R_{k-n-1}]_{\ell}\Big)+n\vert\underline x\vert^{n-2}\underline x\Big(S_{k-n-1}-R_{k-n-1}\Big)\underline x\bigg)\\
+\sum_{\substack{n=1\\n\;\text{odd}}}^{k-1}\vert\underline x\vert^{n-1}\bigg(\underline x\sum_{\ell}\Big((n+1)[R_{k-n-1}]_{\ell}+b_{n,\ell}[S_{k-n-1}]_{\ell}\Big)\\
-\sum_{\ell}\Big((n+1)[R_{k-n-1}]_{\ell}+b_{n,\ell}[S_{k-n-1}]_{\ell}\Big)\underline x\bigg)=0
\end{multline*}
where $a_{n,\ell}=2k+m+\mu_{\ell}-n-2$ and $b_{n,\ell}=2k+m-\mu_{\ell}-n-1$. Lemma \ref{lemfund} now yields 
\begin{alignat*}{2}
R_{k-n}&=S_{k-n},&\quad& n\;\text{odd}\\
[R_{k-n}]_{\ell}&=\lambda_{n,\ell}[S_{k-n}]_{\ell},&\quad& n\;\text{even},
\end{alignat*}
for $n=1,\dots,k-1$. These relations can be assumed also in the case $n=k$. This leads to the desired result.
\end{proof}

\begin{cor}\label{CKxPx}
Let k and n denote non-negative integers. Every two-sided monogenic homogeneous polynomial in $\mathbb R^{m+1}$ can always be written as a finite sum of axial two-sided monogenic polynomials of the form 
\begin{equation}\label{buildblocks2sided}
\mathsf{CK}\big[\alpha_{n,\ell}\vert\underline x\vert^{2n}P_{k,\ell}(\underline x)+\vert\underline x\vert^{2n-2}\underline xP_{k,\ell}(\underline x)\underline x\big],\quad\mathsf{CK}\big[\vert\underline x\vert^{2n}\left(\underline xP_{k,\ell}(\underline x)+P_{k,\ell}(\underline x)\underline x\right)\big],
\end{equation}
where $P_{k,\ell}$ is an $\mathbb R_{0,m}^{(\ell)}$-valued polynomial belonging to $\mathsf{M}_{l}(k)$ $(0\le\ell\le m)$ and 
\[\alpha_{n,\ell}=\displaystyle{-\frac{2k+2n+m-\mu_{\ell}}{2n}}.\]
\end{cor}
\begin{proof}
Observe that Theorem \ref{desc2sided} actually shows that any two-sided monogenic homogeneous polynomial in $\mathbb R^{m+1}$ can be decomposed as a finite sum of left monogenic polynomials of the form (\ref{buildblocks2sided}). We can claim that these polynomials are two-sided monogenic since their restriction to $\mathbb R^{m}$ satisfy the condition $\partial_{\underline x}g=g\partial_{\underline x}$. Finally, with the help of identities (\ref{ident1})-(\ref{ident3-4}), it is easily seen that they are of the form (\ref{AxialTSMF}) and hence are axial two-sided monogenic polynomials.
\end{proof}

\section{Monogenic plane waves leading to axial two-sided monogenics}\label{secc3}

Let $h(x,y)=u(x,y)+iv(x,y)$ be a holomorphic function and assume that $\underline t\in S^{m-1}$ is a fixed unit vector. It is easy to verify that 
\[(\partial_{x_0}+\partial_{\underline x})h(x_0,\theta)=\partial_{x_0}h(x_0,\theta)+\underline t\,\partial_{\theta}h(x_0,\theta)=(1+i\underline t)\partial_{x_0}h(x_0,\theta),\]
where $\theta=\langle\underline x,\underline t\rangle$. Using now the fact that $1+i\underline t$ and $1-i\underline t$ are zero divisors, we get
\[(\partial_{x_0}+\partial_{\underline x})\big((1-i\underline t)h(x_0,\theta)\big)=(1+i\underline t)(1-i\underline t)\partial_{x_0}h(x_0,\theta)=0,\]
which implies that $(1-i\underline t)h(x_0,\theta)$ is a monogenic plane wave. 

Starting with these monogenic plane waves and using Funk-Hecke's formula we will be able to devise a method for constructing axial two-sided monogenic functions. For the reader's convenience we first recall:

\begin{thm}[Funk-Hecke's formula \cite{Hoch}]
Suppose that $\displaystyle{\int_{-1}^1\vert F(t)\vert(1-t^2)^{(m-3)/2}dt<\infty}$ and let $\underline\xi\in S^{m-1}$. If $Y_k(\underline x)$ is a spherical harmonic of degree $k$ in $\mathbb R^m$, then 
\[\int_{S^{m-1}}F(\langle\underline\xi,\underline\eta\rangle)Y_k(\underline\eta)dS(\underline\eta)=\sigma_{m-1}C_k(1)^{-1}Y_{k}(\underline\xi)\int_{-1}^1F(t)C_k(t)\left(1-t^2\right)^{(m-3)/2}dt,\]
where $C_k(t)$ denotes the Gegenbauer polynomial $C^{\nu}_k(t)$ with $\nu=(m-2)/2$ and $\sigma_{m-1}$ is the surface area of the unit sphere $S^{m-2}$ in $\mathbb R^{m-1}$.
\end{thm}

Let $\Delta_{\underline x}=\sum_{j=1}^m\partial_{x_j}^2$ be the Laplacian in $\mathbb R^{m}$ and assume that $P_{k,\ell}$ is an $\mathbb R_{0,m}^{(\ell)}$-valued polynomial belonging to $\mathsf{M}_{l}(k)$. Applying the following identity
\[\Delta_{\underline x}(fg)=(\Delta_{\underline x}f)g+2\sum_{j=1}^m(\partial_{x_j}f)(\partial_{x_j}g)+f(\Delta_{\underline x}g),\]
 one can easily check that polynomials $\underline xP_{k,\ell}, P_{k,\ell}\underline x$ are harmonic and that
 \begin{align*}
\Delta_{\underline x}(\underline xP_{k,\ell}\underline x)&=2\partial_{\underline x}(P_{k,\ell}\underline x)=2\mu_{\ell}P_{k,\ell}\\
\Delta_{\underline x}\left(\vert\underline x\vert^2P_{k,\ell}\right)&=\left(\Delta_{\underline x}\vert\underline x\vert^2\right)P_{k,\ell}+4\sum_{j=1}^mx_j\partial_{x_j}P_{k,\ell}=2(2k+m)P_{k,\ell}.
\end{align*}
The last two equalities enable us to get the classical Fischer decomposition of $\underline xP_{k,\ell}\underline x$, namely:
\begin{equation}\label{Fisch2term}
\underline xP_{k,\ell}\underline x=\left(\underline xP_{k,\ell}\underline x-\vert\underline x\vert^2\frac{\mu_{\ell}}{2k+m}P_{k,\ell}\right)+\vert\underline x\vert^2\frac{\mu_{\ell}}{2k+m}P_{k,\ell}.
\end{equation}

\begin{thm}\label{PlaWavMeth}
The function defined by  
\[I_h(x_0,\underline x)=\frac{1}{\sigma_{m-1}}\int_{S^{m-1}}h(x_0,\langle\underline x,\underline t\rangle)(1-i\underline t)P_{k,\ell}(\underline t)(1-i\underline t)dS(\underline t)\]
is axial two-sided monogenic with
\begin{multline*}
A_h(x_0,r)=\frac{r^{-k}}{2k+m}\left((2k+m-\mu_{\ell})C_k(1)^{-1}\int_{-1}^1h(x_0,rt)C_k(t)\left(1-t^2\right)^{(m-3)/2}dt\right.\\
\left.+\mu_{\ell}\,C_{k+2}(1)^{-1}\int_{-1}^1h(x_0,rt)C_{k+2}(t)\left(1-t^2\right)^{(m-3)/2}dt\right),
\end{multline*}
\[B_h(x_0,r)=C_h(x_0,r)=-ir^{-k-1}C_{k+1}(1)^{-1}\int_{-1}^1h(x_0,rt)C_{k+1}(t)\left(1-t^2\right)^{(m-3)/2}dt,\]
\[D_h(x_0,r)=-r^{-k-2}C_{k+2}(1)^{-1}\int_{-1}^1h(x_0,rt)C_{k+2}(t)\left(1-t^2\right)^{(m-3)/2}dt.\]
\end{thm}
\begin{proof}
It is clear that for any $\underline t\in S^{m-1}$ the function $h(x_0,\langle\underline x,\underline t\rangle)(1-i\underline t)P_{k,\ell}(\underline t)(1-i\underline t)$ is two-sided monogenic and hence so is the function $I_h(x_0,\underline x)$. We thus only need to show that it may be written as
\[I_h=A_hP_{k,\ell}+B_h\underline xP_{k,\ell}+C_hP_{k,\ell}\underline x+D_h\underline x P_{k,\ell}\underline x.\]
In order to perform this task we must compute integrals of the form
\[\frac{1}{\sigma_{m-1}}\int_{S^{m-1}}h(x_0,\langle\underline x,\underline t\rangle)F(\underline t)dS(\underline t),\]
where $F(\underline t)$ can be equal to $P_{k,\ell}(\underline t)$, $\underline tP_{k,\ell}(\underline t)$, $P_{k,\ell}(\underline t)\underline t$ or $\underline tP_{k,\ell}(\underline t)\underline t$. These integrals shall be denoted by $I_1$, $I_2$, $I_3$ and $I_4$. We then have that 
\[I_h=I_1-iI_2-iI_3-I_4.\]
The first three integrals may be computed directly by applying Funk-Hecke's formula since  $P_{k,\ell}(\underline t)$, $\underline tP_{k,\ell}(\underline t)$ and $P_{k,\ell}(\underline t)\underline t$ are harmonic polynomials. Indeed, writing $\underline x$ in polar coordinates, i.e. $\underline x=r\underline\omega$, we obtain
\begin{align*}
I_1&=P_{k,\ell}(\underline\omega)C_k(1)^{-1}\int_{-1}^1h(x_0,rt)C_k(t)\left(1-t^2\right)^{(m-3)/2}dt\\
I_2&=\underline\omega P_{k,\ell}(\underline\omega)C_{k+1}(1)^{-1}\int_{-1}^1h(x_0,rt)C_{k+1}(t)\left(1-t^2\right)^{(m-3)/2}dt,\\
I_3&=P_{k,\ell}(\underline\omega)\underline\omega C_{k+1}(1)^{-1}\int_{-1}^1h(x_0,rt)C_{k+1}(t)\left(1-t^2\right)^{(m-3)/2}dt.
\end{align*}
Finally, from (\ref{Fisch2term}) and using Funk-Hecke's formula we also get 
\begin{multline*}
I_4=\left(\underline\omega P_{k,\ell}(\underline\omega)\underline\omega-\frac{\mu_{\ell}}{2k+m}P_{k,\ell}(\underline\omega)\right)C_{k+2}(1)^{-1}\\
\times\int_{-1}^1h(x_0,rt)C_{k+2}(t)\left(1-t^2\right)^{(m-3)/2}dt\\
+\frac{\mu_{\ell}}{2k+m}P_{k,\ell}(\underline\omega)C_k(1)^{-1}\int_{-1}^1h(x_0,rt)C_k(t)\left(1-t^2\right)^{(m-3)/2}dt,
\end{multline*}
which completes the proof.
\end{proof}
\noindent
In the next examples we compute $I_h$ for the cases $h(x,y)=e^{x+iy}$ and $h(x,y)=(x+iy)^n$.\vspace{0.2cm} 

\noindent
{\bf Example 1.} An axial two-sided monogenic function of exponential type was obtained in \cite{DSo} by assuming the existence of a solution of (\ref{Veq2sided}) of the form 
\[A(x_0,r)=e^{x_0}a(r),\;B(x_0,r)=e^{x_0}b(r),\;D(x_0,r)=e^{x_0}d(r).\]
This assumption led to an ordinary differential equation of second order for $b(r)$ which could be solved by means of the Bessel function of the first kind $J_{k+m/2}(r)$, namely
\begin{equation*}
b(r)=r^{-k-\frac{m}{2}}J_{k+\frac{m}{2}}(r).
\end{equation*}
From this it easily follows that 
\begin{align*}
d(r)&=r^{-k-\frac{m}{2}-1}J_{k+\frac{m}{2}+1}(r),\\
a(r)&=\big(2k+m-\mu_{\ell}\big)b(r)-r^2d(r).
\end{align*}
We will now show that this particular solution of system (\ref{Veq2sided}) can be derived from Theorem \ref{PlaWavMeth} by assuming $h(x,y)=e^{x+iy}$. In order to do this we shall use the following equalities
\begin{equation*}
C_k^\nu(1)=\frac{\Gamma(2\nu+k)}{k!\,\Gamma(2\nu)},\quad\Gamma\left(\frac{n}{2}\right)=\sqrt{\pi}\frac{(n-2)!!}{2^{(n-1)/2}},
\end{equation*}
\[\int_{-1}^1e^{iat}C_k^\nu(t)\left(1-t^2\right)^{\nu-1/2}dt=\frac{\pi\,2^{1-\nu}i^k\Gamma(2\nu+k)}{k!\,\Gamma(\nu)}a^{-\nu}J_{k+\nu}(a),\]
where $\Gamma$ denotes the Gamma function and  $n!!$ the double factorial of $n$ (see e.g. \cite{GraRy}). It follows that 
\begin{multline*}
r^{-k}C_k(1)^{-1}\int_{-1}^1e^{irt}C_k(t)\left(1-t^2\right)^{(m-3)/2}dt\\
=\sqrt{2\pi}\,(m-3)!!\,i^kr^{-(k+m/2-1)}J_{k+m/2-1}(r),
\end{multline*}
from which we immediately get
\[B_h(x_0,r)=\sqrt{2\pi}\,(m-3)!!\,i^ke^{x_0}b(r),\quad D_h(x_0,r)=\sqrt{2\pi}\,(m-3)!!\,i^ke^{x_0}d(r).\]
For computing $A_h(x_0,r)$ we also need the recurrence relation  
\[\frac{2\nu}{r}J_\nu(r)=J_{\nu-1}(r)+J_{\nu+1}(r)\]
to obtain $A_h(x_0,r)=\sqrt{2\pi}\,(m-3)!!\,i^ke^{x_0}a(r)$. Therefore
\begin{multline*}
I_h(x_0,\underline x)=\sqrt{2\pi}\,(m-3)!!\,i^ke^{x_0}\Big(a(r)P_{k,\ell}(\underline x)+b(r)\underline xP_{k,\ell}(\underline x)\\
+b(r)P_{k,\ell}(\underline x)\underline x+d(r)\underline x P_{k,\ell}(\underline x)\underline x\Big)
\end{multline*}
for $h(x,y)=e^{x+iy}$.\vspace{0.2cm} 

\noindent
{\bf Example 2.} Other two interesting choices of $h$ are provided by the holomorphic functions 
\[h(x,y)=(x+iy)^{k+2n},\quad h(x,y)=(x+iy)^{k+2n+1}\]   
because they yield the basic axial two-sided monogenic polynomials (\ref{buildblocks2sided}). 

Let us first consider the case $h(x,y)=(x+iy)^{k+2n}$. Note that for this case $h(0,rt)C_{k+1}(t)$ is odd as a function of $t$ and therefore $B_h(0,r)=C_h(0,r)=0$. For the computation of $A_h(0,r)$ and $D_h(0,r)$ we use the following identity   
\[\int_{0}^1t^{k+2\rho}C_k^\nu(t)\left(1-t^2\right)^{\nu-1/2}dt=\frac{\Gamma(2\nu+k)\Gamma(2\rho+k+1)\Gamma\left(\nu+\frac{1}{2}\right)\Gamma\left(\rho+\frac{1}{2}\right)}{2^{k+1}\Gamma(2\nu)\Gamma(2\rho+1)\,k!\,\Gamma(k+\nu+\rho+1)}\]
and we can conclude that
\begin{multline*}
I_h(x_0,\underline x)=\frac{(-1)^{n+1}\sqrt{2\pi}\,(k+2n)!(m-3)!!\,i^k}{(2n-2)!!(2k+2n+m)!!}\\
\times\mathsf{CK}\big[\alpha_{n,\ell}\vert\underline x\vert^{2n}P_{k,\ell}(\underline x)+\vert\underline x\vert^{2n-2}\underline xP_{k,\ell}(\underline x)\underline x\big](x_0,\underline x).
\end{multline*}
A similar analysis can be made for the case $h(x,y)=(x+iy)^{k+2n+1}$ to obtain 
\begin{multline*}
I_h(x_0,\underline x)=\frac{(-1)^{n}\sqrt{2\pi}\,(k+2n+1)!(m-3)!!\,i^k}{(2n)!!(2k+2n+m)!!}\\
\times\mathsf{CK}\big[\vert\underline x\vert^{2n}\left(\underline xP_{k,\ell}(\underline x)+P_{k,\ell}(\underline x)\underline x\right)\big](x_0,\underline x).
\end{multline*}

\begin{rem}
In view of Corollary \ref{CKxPx} it does follow that every two-sided monogenic homogeneous polynomial in $\mathbb R^{m+1}$ can always be written as a finite sum of functions $I_h$ where $h(x,y)=(x+iy)^{k+2n}$ or $h(x,y)=(x+iy)^{k+2n+1}$.
\end{rem}

\section{A characterization in terms of derivatives of axial left monogenic functions}\label{secc4}

Proposition \ref{caract1} gives a characterization of the axial two-sided monogenic functions. The goal in this section is to offer an alternative description by  showing the connection between these functions and the axial left monogenic functions. 

Suppose that $P_{k,\ell}$ is an $\mathbb R_{0,m}^{(\ell)}$-valued polynomial belonging to $\mathsf{M}_{l}(k)$. If 
\[\left(M(x_0,r)+\displaystyle{\frac{\underline x}{r}}\,N(x_0,r)\right)P_{k,\ell}(\underline x)\] 
is axial left monogenic, then it is clear that 
\[\left[\left(M(x_0,r)+\frac{\underline x}{r}\,N(x_0,r)\right)P_{k,\ell}(\underline x)\right](\partial_{x_0}-\partial_{\underline x})\]
is also right monogenic. This function is moreover of the form (\ref{AxialTSMF}) with
\begin{equation}\label{rleft2sided}
A=\partial_{x_0}M-\mu_{\ell}\frac{N}{r},\quad B=\frac{\partial_{x_0}N}{r},\quad C=-\frac{\partial_{r}M}{r},\quad D=-\frac{\partial_{r}\left(N/r\right)}{r}
\end{equation}
and hence is axial two-sided monogenic. Observe that $B=C$, which follows from the second equation of (\ref{VeEq}).

It is natural to ask whether every axial two-sided monogenic function can be obtained in this way. 

\begin{thm}
Let $F=AP_{k,\ell}+B\underline xP_{k,\ell}+CP_{k,\ell}\underline x+D\underline x P_{k,\ell}\underline x$ be an axial two-sided monogenic function defined in an open neighbourhood of
\[\Omega=\left\{(x_0,\underline x)\in\mathbb R^{m+1}:\;(x_0,r)\in [a_1,b_1]\times[a_2,b_2]\subset\mathbb R^2,\;a_2>0\right\}.\]
There exists an axial left monogenic function $\left(M+\displaystyle{\frac{\underline x}{r}}\,N\right)P_{k,\ell}$ such that
\[F(x_0,\underline x)-\left[\left(M(x_0,r)+\frac{\underline x}{r}\,N(x_0,r)\right)P_{k,\ell}(\underline x)\right](\partial_{x_0}-\partial_{\underline x})=cP_{k,\ell}(\underline x),\]
where $c$ is a real constant.
\end{thm}
\begin{proof}
On account of (\ref{rleft2sided}) we need to find solutions $M$, $N$ to the system 
\begin{align*}
\partial_{r}M&=-rB\\
\partial_{r}\left(N/r\right)&=-rD
\end{align*}
that satisfy the Vekua system (\ref{VeEq}). Thus we have
\[M(x_0,r)=-\int_{a_2}^rtB(x_0,t)dt+\alpha(x_0),\]
\[N(x_0,r)=r\left(-\int_{a_2}^rtD(x_0,t)dt+\beta(x_0)\right).\]
Using the last two equations of (\ref{Veq2sided}) we obtain
\begin{align*}
\partial_{x_0}M=-\int_{a_2}^r\left(t^2\partial_{t}D(x_0,t)+(2k+m+2)tD(x_0,t)\right)dt+\alpha^{\prime}(x_0)\\
=-(2k+m)\int_{a_2}^rtD(x_0,t)dt-\big(t^2D(x_0,t)\big)\big\vert_{t=a_2}^{t=r}+\alpha^{\prime}(x_0),
\end{align*}
\[\partial_{x_0}N=r\left(\int_{a_2}^r\partial_{t}B(x_0,t)dt+\beta^{\prime}(x_0)\right)=r\left(B(x_0,t)\big\vert_{t=a_2}^{t=r}+\beta^{\prime}(x_0)\right).\]
Hence
\[\partial_{x_0}M-\partial_{r}N=\frac{2k+m-1}{r}N+\alpha^{\prime}(x_0)-(2k+m)\beta(x_0)+a_2^2D(x_0,a_2)\]
and 
\[\partial_rM+\partial_{x_0}N=r\left(\beta^{\prime}(x_0)-B(x_0,a_2)\right).\]
Therefore, $M$ and $N$ satisfy the Vekua system (\ref{VeEq}) if and only if 
\begin{align*}
\alpha^{\prime}(x_0)-(2k+m)\beta(x_0)&=-a_2^2D(x_0,a_2)\\
\beta^{\prime}(x_0)&=B(x_0,a_2).
\end{align*}
Thus, it is possible to find an axial left monogenic function $\left(M+\displaystyle{\frac{\underline x}{r}}\,N\right)P_{k,\ell}$ such that
\[F(x_0,\underline x)-\left[\left(M(x_0,r)+\frac{\underline x}{r}\,N(x_0,r)\right)P_{k,\ell}(\underline x)\right](\partial_{x_0}-\partial_{\underline x})=c(x_0,r)P_{k,\ell}(\underline x),\]
where $c(x_0,r)$ is an $\mathbb R$-valued function. The monogenicity of the left-hand side of the last equality implies that function $c(x_0,r)$ is a constant.  
\end{proof}

\subsection*{Acknowledgments}

D. Pe\~na Pe\~na acknowledges the support of a Postdoctoral Fellowship given by Istituto Nazionale di Alta Matematica (INdAM) and cofunded by Marie Curie actions.


\begin{thebibliography}{99}

\bibitem{ABoDS} R. Abreu Blaya, J. Bory Reyes, R. Delanghe and F. Sommen, \textit{Harmonic multivector fields and the Cauchy integral decomposition in Clifford analysis.} Bull. Belg. Math. Soc. Simon Stevin 11 (2004), no. 1, 95--110.

\bibitem{BDS} F. Brackx, R. Delanghe and F. Sommen, \textit{Clifford analysis.} Research Notes in Mathematics, 76, Pitman (Advanced Publishing Program), Boston, MA, 1982.

\bibitem{Cl} W. K. Clifford, \textit{Applications of Grassmann's Extensive Algebra.} Amer. J. Math. 1 (1878), no. 4, 350--358.

\bibitem{CS4} F. Colombo, I. Sabadini, F. Sommen and D. C. Struppa, \textit{Analysis of Dirac systems and computational algebra.} Progress in Mathematical Physics, 39. Birkh\"auser Boston, Inc., Boston, MA, 2004. 

\bibitem{DLaS} R. Delanghe, R. L\'avi\v cka and V. Sou\v cek, \textit{The Gelfand-Tsetlin bases for Hodge-de Rham systems in Euclidean spaces.} Math. Methods Appl. Sci. 35 (2012), no. 7, 745--757.

\bibitem{DSS} R. Delanghe, F. Sommen and V. Sou\v cek, \textit{Clifford algebra and spinor-valued functions.} Mathematics and its Applications, 53, Kluwer Academic Publishers Group, Dordrecht, 1992.

\bibitem{GM} J. Gilbert and M. Murray, \textit{Clifford algebras and Dirac operators in harmonic analysis.} Cambridge University Press, Cambridge, 1991. 

\bibitem{GraRy} I. S. Gradshteyn and I. M. Ryzhik, \textit{Table of integrals, series, and products.} Seventh edition. Elsevier/Academic Press, Amsterdam, 2007.

\bibitem{GuSp} K. G\"urlebeck and W. Spr\"ossig, \textit{Quaternionic and Clifford calculus for physicists and engineers.} Wiley and Sons Publications, Chichester, 1997.

\bibitem{Hoch} H. Hochstadt, \textit{The functions of mathematical physics.} Dover Publications, Inc., New York, 1986.

\bibitem{Lo} P. Lounesto, \textit{Clifford algebras and spinors.} London Mathematical Society Lecture Note Series, 286. Cambridge University Press, Cambridge, 2001.

\bibitem{LB} P. Lounesto and P. Bergh, \textit{Axially symmetric vector fields and their complex potentials.} Complex Variables Theory Appl. 2 (1983), no. 2, 139--150.

\bibitem{DSo} D. Pe\~{n}a Pe\~{n}a and F. Sommen, \textit{Vekua-type systems related to two-sided monogenic functions.} Complex Anal. Oper. Theory 6 (2012), no. 2, 397--405. 

\bibitem{So1} F. Sommen, \textit{A product and an exponential function in hypercomplex function theory.} Applicable Anal. 12 (1981), no. 1, 13--26. 

\bibitem{S1} F.  Sommen, \textit{Plane elliptic systems and monogenic functions in symmetric domains.} Rend. Circ. Mat. Palermo (2) 1984, no. 6, 259--269.

\bibitem{So2} F. Sommen, \textit{Plane waves, biregular functions and hypercomplex Fourier analysis.} Rend. Circ. Mat. Palermo (2) no. 9 (1985), 205--219.

\bibitem{So3} F. Sommen, \textit{Plane wave decompositions of monogenic functions.} Ann. Polon. Math. 49 (1988), no. 1, 101--114.

\bibitem{S2} F. Sommen, \textit{Special functions in Clifford analysis and axial symmetry.} J. Math. Anal. Appl. 130 (1988), no. 1, 110--133.

\bibitem{S3} F. Sommen, \textit{On a generalization of Fueter's theorem.} Z. Anal. Anwendungen 19 (2000), no. 4, 899--902.

\bibitem{Ve} I. N. Vekua, \textit{Generalized analytic functions.} Pergamon Press, London, 1962.

\end{thebibliography}
\end{document}